\documentclass[11pt,reqno]{amsart}
\usepackage{amsmath, amssymb,amscd, amsthm, tikz-cd, mathrsfs}

\usepackage[alphabetic,lite]{amsrefs}

\usepackage{calligra}

 \newtheorem{theorem}{Theorem}[section]
 \newtheorem{lemma}[theorem]{Lemma}
 \newtheorem{corollary}[theorem]{Corollary}
 \newtheorem{prop}[theorem]{Proposition}

 \theoremstyle{definition}
 \newtheorem{remark}[theorem]{Remark}
 
 \theoremstyle{definition}
 \newtheorem{example}[theorem]{Example}

 \newtheorem{claim}[theorem]{Claim}

\newtheorem{goal}[theorem]{Goal}

\numberwithin{equation}{section}

\makeatletter
\newtheorem*{rep@theorem}{\rep@title}
\newcommand{\newreptheorem}[2]{%
\newenvironment{rep#1}[1]{%
 \def\rep@title{#2 \ref{##1}}%
 \begin{rep@theorem}}%
 {\end{rep@theorem}}}
\makeatother

\newreptheorem{theorem}{Theorem}

\newcommand{\C}{\mathbb{C}}
\newcommand{\GL}{\operatorname{GL}}

\newcommand{\D}{\mathcal{D}}

\newcommand{\X}{\mathscr{X}}
\newcommand{\Z}{\mathscr{Z}}
\newcommand{\Y}{T}
\newcommand{\IC}{IC}

\newcommand{\Ext}{\operatorname{Ext}}

\newcommand{\codim}{\operatorname{codim}}

\newcommand{\dom}{\operatorname{dom}}

\renewcommand{\mod}{\operatorname{mod}}

\newcommand{\bS}{\mathbb{S}}
\newcommand{\bw}{\bigwedge}
\newcommand{\Sym}{\operatorname{Sym}}
\renewcommand{\det}{\operatorname{det}}

\newcommand{\fD}{\mathfrak{D}}

\newcommand{\lam}{\lambda}

\newcommand{\tn}{\textnormal}
\newcommand{\defi}[1]{{\upshape\sffamily #1}}

\newcommand{\lra}{\longrightarrow}
\newcommand{\mc}[1]{\mathcal{#1}}

\title{Mixed Hodge structure on Local Cohomology with support in determinantal varieties}

\author{Michael Perlman}
\address{School of Mathematics, University of Minnesota, Minneapolis, MN, 55455}
\email{mperlman@umn.edu}

\begin{document}

\maketitle 

\begin{abstract}
We employ the inductive structure of determinantal varieties to calculate the mixed Hodge module structure of local cohomology modules with determinantal support. We show that the weight of a simple composition factor is uniquely determined by its support and cohomological degree. As a consequence, we obtain the equivariant structure of the Hodge filtration on each local cohomology module. Finally, as an application, we provide a formula for the generation level of the Hodge filtration on these modules.
\end{abstract}
\section{Introduction}

Given a smooth complex variety $X$, and a closed subvariety $Z\subseteq X$, the local cohomology sheaves $\mathcal{H}^j_{Z}(\mc{O}_X)$ are holonomic $\D_X$-modules, where $\D_X$ is the sheaf of algebraic differential operators. Furthermore, the sheaves $\mathcal{H}^j_{Z}(\mc{O}_X)$ are functorially endowed with structures of mixed Hodge modules \cite{saito90}, implying that they are equipped with two increasing filtrations: the \defi{Hodge filtration} $F_{\bullet}(\mathcal{H}^j_{Z}(\mc{O}_X))$, an infinite filtration by coherent $\mc{O}_X$-modules; and the \defi{weight filtration} $W_{\bullet}(\mathcal{H}^j_{Z}(\mc{O}_X))$, a finite filtration by holonomic $\D_{X}$-modules. 

When $Z$ is a divisor, the data of the Hodge filtration on the module $\mc{H}^1_Z(\mc{O}_X)$ is equivalent to that of the Hodge ideals \cite{MP1}. In this case, there are numerous connections between the behavior of the Hodge filtration (e.g. the jumps and generation level) and invariants of singularities arising from birational geometry, including the multiplier ideals and minimal exponent \cites{MPP2, MP3}. In the setting of higher-codimension $Z$, the Hodge filtration on local cohomology detects Du Bois singularities \cite[Theorem C]{MP4}, and it determines when a complete intersection is non-singular \cite[Corollary 9.6]{MP4}, among other things (see \cite{MP4}).

Despite the recent interest in the mixed Hodge module structure on local cohomology, very few explicit examples are known. In the higher-codimension setting there are no calculations of the Hodge filtration on local cohomology except the case when $Z$ is smooth \cite[Example 3.8]{MP4} and the case when $Z$ is defined by a monomial ideal \cite[Example 3.11]{MP4}. In fact, there are not many extensive examples of mixed Hodge module structures being understood in general, outside the case of GKZ systems \cite{gkz2, gkz}. In this article, and in previous joint work with Claudiu Raicu \cite{perlmanraicu}, we aim to explicate the Hodge and weight filtrations on local cohomology  when $Z$ is a generic determinantal variety, with the hope that such calculations will lead to insights regarding the mixed Hodge module structure of local cohomology in general.  

We let $\X=\C^{m\times n}$ be the space of $m\times n$ generic matrices, with $m\geq n$, endowed with the action of the group $\GL=\GL_m(\C)\times \GL_n(\C)$ via row and column operations. For $0\leq q \leq n$ we let $\Z_q\subseteq \X$ denote the determinantal variety of matrices of rank $\leq q$. The $\D$-module structure of the local cohomology modules $H^j_{\Z_q}(\mc{O}_{\X})$ is well understood \cite{raicu2014local, raicu2014locals, raicu2016local, lHorincz2018iterated}. In particular, their simple composition factors are known \cite[Main Theorem]{raicu2016local}, which are among $D_0,\cdots , D_n$, where $D_p=\mathcal{L}(\Z_p,\X)$ is the intersection homology module associated to the trivial local system on $\Z_p\setminus \Z_{p-1}$. With these explicit formulas in hand (recalled at the beginning of Section \ref{SecWeight}), our first result below completely describes the weight filtration.  

We write $\mc{O}_{\X}^H$ for the trivial Hodge module overlying $\mc{O}_{\X}$. Given a simple composition factor $\mc{M}$ of $H^j_{\Z_q}(\mc{O}^H_{\X})$, we say that $\mc{M}$ has weight $w$ if it is a summand of $\tn{gr}^W_wH^j_{\Z_q}(\mc{O}^H_{\X})$, where $\tn{gr}_{\bullet}^W(-)$ denotes the associated graded functor with respect to the weight filtration $W_{\bullet}$.

The following theorem demonstrates that the weight of pure Hodge module overlying a copy of $D_p$ in local cohomology is uniquely determined by its cohomological degree. 

\begin{theorem}\label{Othm}
Let $0\leq p\leq q< n\leq m$ and $j\geq 0$. If $\mc{M}$ is a simple composition factor of $H^j_{\Z_q}(\mc{O}^H_{\X})$ that overlies a copy of $D_p$, then $\mc{M}$ has weight $mn+q-p+j$. 
 \end{theorem}

The case $q=n-1$ was established in \cite[Theorem 1.3, Theorem 1.5]{perlmanraicu}, and the argument here gives an alternate proof of \cite[Theorem 1.5]{perlmanraicu} without appealing to the Decomposition Theorem. Theorem \ref{Othm} is a consequence of our main result, Theorem \ref{main}, which describes the mixed Hodge module structure on  local cohomology of any pure Hodge module overlying a simple module $D_p$.

In Section \ref{LCP} we explain the choice of Hodge structure on local cohomology implicit in our discussion. For now, we mention that it is determined functorially by pushing forward the trivial Hodge module $\mc{O}_U^H$ from the complement $U=\X\setminus \Z_q$, and so by the general theory $H^j_{\Z_q}(\mc{O}^H_{\X})$ has weight $\geq mn+j-1$ \cite[Proposition 1.7]{saito89}. Theorem \ref{Othm} demonstrates that this bound is not sharp for any $j$ in this case. It would be interesting to find sharper lower bounds in general, perhaps depending on the type of singularities that the variety possesses. From Theorem \ref{Othm} one sees an inverse relation between the weight of a composition factor and the dimension of its support. A somewhat similar correlation has been observed for weights in GKZ systems \cite[Proposition 3.6(2)]{gkz}, though weight is not completely determined by support dimension in that case. 

One may calculate the weight filtration on any local cohomology module $H^j_{\Z_q}(\mc{O}^H_{\X})$ using \cite[Main Theorem]{raicu2016local} and Theorem \ref{Othm}. In the following examples, we express each module as a class in the Grothendieck group of $\GL$-equivariant holonomic $\D$-modules (see Section \ref{equivar}).

\begin{example}\label{Ex1}
Let $m=n=4$ and $q=2$. The three nonzero local cohomology modules are:
$$
\big[H^4_{\Z_2}(\mc{O}_{\X})\big]=[D_2]+[D_1]+[D_0],\;\;\big[H^6_{\Z_2}(\mc{O}_{\X})\big]=[D_1]+[D_0],\;\;\big[H^8_{\Z_2}(\mc{O}_{\X})\big]=[D_0].
$$	
Theorem \ref{Othm} asserts that if we endow $\mc{O}_{\X}$ with the trivial pure Hodge structure, then each  $D_p$ in cohomological degree $j$ above underlies a pure Hodge module of weight $18-p+j$. Thus, the weights of the above simple composition factors are (from left to right): $20$, $21$, $22$, $23$, $24$, $26$.
\end{example}
 
We carry out a larger example on non-square matrices.

\begin{example}\label{Ex2}
Let $m=7$, $n=5$, and $q=3$. The seven nonzero local cohomology modules are:
$$
\big[H^8_{\Z_3}(\mc{O}_{\X})\big]=[D_3],\;\; \big[H^{10}_{\Z_3}(\mc{O}_{\X})\big]=[D_2],\;\;\big[H^{12}_{\Z_3}(\mc{O}_{\X})\big]=[D_2]+[D_1],
$$
$$
\big[H^{14}_{\Z_3}(\mc{O}_{\X})\big]=\big[H^{16}_{\Z_3}(\mc{O}_{\X})\big]=[D_1]+[D_0],\;\; \big[H^{18}_{\Z_3}(\mc{O}_{\X})\big]=\big[H^{20}_{\Z_3}(\mc{O}_{\X})\big]=[D_0].
$$
If $\mc{O}_{\X}$ is endowed with the trivial pure Hodge structure, then the above simple composition factors underlie Hodge modules of the following weights (from left to right): $43$, $46$, $48$, $49$, $51$, $52$, $53$, $54$, $56$, $58$.
\end{example}

One can show that the maximal weight of a composition factor of $H^{\bullet}_{\Z_q}(\mc{O}^H_{\X})$ is $2mn-q(q+1)$, the weight of $D_0$ in the largest degree in which it appears. This upper bound on weight resembles the case of GKZ systems, where weight is bounded above by twice the dimension of the relevant torus \cite[Proposition 3.6(1)]{gkz}.

We now discuss the Hodge filtration. The possible Hodge filtrations on each simple module $D_p$ are uniquely determined by weight, and are calculated in \cite[Theorem 3.1]{perlmanraicu}. As a consequence, we obtain the Hodge filtration on each local cohomology module from our knowledge of the weight filtration. The group $\GL=\GL_m(\C)\times \GL_n(\C)$ acts on $\X$, preserving each determinantal variety, and inducing the structure of a $\GL$-representation on each piece of the Hodge filtration on local cohomology with determinantal support. As such, we express the $\GL$-equivariant structure of the Hodge filtration $F_{\bullet}$ via multisets $\mathfrak{W}(F_{k}(H^j_{\Z_q}(\mc{O}_{\X})))$ of \defi{dominant weights} $\lambda=(\lam_1\geq \cdots \geq \lam_n)\in \mathbb{Z}^n$, encoding the irreducible representations that appear and their multiplicities (see Section \ref{SecGLRep}). In the following statement, we write $c_p=\codim \Z_p=(m-p)(n-p)$.

\begin{corollary}\label{hodgecor}
For $k\in \mathbb{Z}$ and $0\leq p\leq n$ we let
\begin{equation}\label{defup}
\mathfrak{D}^p_k=\big\{\lam \in \mathbb{Z}^n_{\dom}:\lam_p\geq p-n,\;\lam_{p+1}\leq p-m,\;\lam_{p+1}+\cdots +\lam_n\geq-k-c_p\big\}.
\end{equation}
Let $0\leq q< n\leq m$ and $j\geq 0$. The $k$-th piece of the Hodge filtration on local cohomology is determined by the following multiset of dominant weights:
\begin{equation}\label{hodgeH}
\mathfrak{W}\big(F_k\big(H^j_{\Z_q}\left(\mc{O}_{\X}^H\right)\big)\big)=\bigsqcup_{p=0}^q \left({\large \mathfrak{D}}^p_{k-(c_p+p-q-j)/2}\right)^{\sqcup \; a_p},
\end{equation}
where $a_p$ is the multiplicity of $D_p$ as a simple composition factor of $H^j_{\Z_q}(\mc{O}_{\X})$.
\end{corollary}

The case $q=n-1$ appeared in \cite{perlmanraicu}, where it was used to calculate the Hodge ideals for the determinant hypersurface. In Section \ref{SecHodgeFilt}, we elaborate on Corollary \ref{hodgecor} and explain how to deduce it from Theorem \ref{Othm} and \cite{perlmanraicu}. We say more about the $\mc{O}_{\X}$-module structure of the Hodge filtration in Remark \ref{ClaudiuRemark} and Remark \ref{HodgeIdealRemark}. 

For now, we discuss what one may deduce about the Hodge filtration from Corollary \ref{hodgecor}. Each $\fD^p_k$ in (\ref{hodgeH}) arises from the induced Hodge filtration on a composition factor $D_p$, and Corollary \ref{hodgecor} asserts that the filtration on each $D_p$ in cohomological degree $j$ is the same. It follows from (\ref{defup}) and (\ref{hodgeH}) that the first nonzero level of Hodge filtration on each $D_p$ is $(c_p+p-q-j)/2$. In particular, as cohomological degree increases, the starting level of Hodge filtration on each $D_p$ decreases. As the sets $\fD^p_k$ and $\fD^{r}_l$ are disjoint for all  $p\neq r$ and all indices $k,l$, there is no ambiguity in (\ref{hodgeH}), and it completely describes the Hodge filtration on a local cohomology module. 

As an application, we determine the generation level of the Hodge filtration (see Section \ref{SecGenLev}).
\begin{corollary}\label{newgenlevel}
If $s$ is minimal such that $a_{s}\neq 0$ in (\ref{hodgeH}), then the generation level of the Hodge filtration on $H^j_{\Z_q}(\mc{O}^H_{\X})$ is $(c_{s}+s-q-j)/2$. If $m=n$, then $s=0$, so the generation level is $(n^2-q-j)/2$.	
\end{corollary}

Below is an example which illustrates Corollary \ref{hodgecor} and Corollary \ref{newgenlevel}.

\begin{example}
Continuing Example \ref{Ex1}, for $k\in \mathbb{Z}$ we have	
$$
\mathfrak{W}\big(F_k\big(H^4_{\Z_2}\big( \mc{O}^H_{\X}\big)\big)\big)=\fD^2_k\sqcup \fD^1_{k-2}\sqcup \fD^0_{k-5},\;\;\mathfrak{W}\big(F_k\big(H^6_{\Z_2}\big( \mc{O}^H_{\X}\big)\big)\big)=\fD^1_{k-1}\sqcup \fD^0_{k-4},
$$
$$
\mathfrak{W}\big(F_k\big(H^8_{\Z_2}\big( \mc{O}^H_{\X}\big)\big)\big)=\fD^0_{k-3}.
$$
The induced filtrations on each of the simple composition factors start in the following levels (from left to right): $0$, $2$, $5$, $1$, $4$, $3$, and the generation level of the Hodge filtration on each of the three nonzero local cohomology modules is $5$, $4$, $3$, respectively.
\end{example}

\subsection*{Strategy and Organization} We summarize our strategy to prove Theorem \ref{main} (and Theorem \ref{Othm}). The proof proceeds by induction on $n\geq 1$.

\noindent \textbf{Step 1.} In Section \ref{SecIndHodge} we employ the inductive structure of determinantal varieties to relate the mixed Hodge module structure of local cohomology on $\X$ with support in $\Z_q$ to that on smaller matrices $\C^{(m-1)\times (n-1)}$ with support matrices of rank $\leq q-1$. By the inductive hypothesis, this allows us to reduce to the problem of verifying that each copy of $D_0$ in  local cohomology underlies a Hodge module of the desired weight.

\noindent \textbf{Step 2.} In Section \ref{SecProof} we proceed by induction on $q\geq 0$, completing the inductive step by examining weights in some Grothendieck spectral sequences for local cohomology.

In Section \ref{SecHodgeFilt} we discuss how to deduce the Hodge filtration on each local cohomology module from Theorem \ref{main} and \cite{perlmanraicu}. As an application, we determine the generation level.


\section{Preliminaries}\label{Prelim}

In this section we establish notation, and review some relevant background regarding functors on Hodge modules, local cohomology, and equivariant $\D$-modules. All of our $\D$-modules are left $\D$-modules. 
\subsection{$\D$-modules, Hodge modules, and functors}\label{functors}Let $X$ be a smooth complex variety of dimension $d_{X}$, with sheaf of algebraic differential operators $\D_{X}$. We write $\tn{D}_h^b(\D_{X})$ for the bounded derived category of holonomic  $\D_{X}$-modules, and we write $\tn{MHM}(X)$ for the category of algebraic mixed Hodge modules on $X$ (see \cite[Section 8.3.3]{hotta2007d}), with $\tn{D}^b\tn{MHM}(X)$ the corresponding bounded derived category. 

Given an irreducible closed subvariety $Z\subseteq X$, we write $\mc{L}(Z,X)$ for the \defi{intersection homology $\D$-module} associated to the trivial local system on the regular locus $Z_{\tn{reg}}\subseteq Z$ \cite[Definition 3.4.1]{hotta2007d}. 

We write $\IC^H_Z$ for the pure Hodge module associated to the trivial variation of Hodge structure on $Z_{\tn{reg}}$ \cite[Section 8.3.3(m13)]{hotta2007d}, which has weight $d_Z$. For a mixed Hodge module $\mc{M}=(M,F_{\bullet},W_{\bullet})$ and $k\in \mathbb{Z}$, we write $\mc{M}(k)=(M,F_{\bullet-k},W_{\bullet+2k})$ for its $k$-th \defi{Tate twist} \cite[Section 8.3.3(m5)]{hotta2007d}. For example, $\IC^H_Z(k)$ has weight $d_Z-2k$. The modules $\IC^H_Z(k)$ provide a complete list of polarizable pure Hodge modules that may overlie the $\D$-module $\mc{L}(Z,X)$ \cite[Section 8.3.3(m13)]{hotta2007d}.  

 Let $f:X\to Y$ be a morphism between smooth complex varieties, and let $M\in \tn{D}_h^b(\D_{X})$ and $N\in \tn{D}_h^b(\D_{Y})$. We write the following for the direct and inverse image functors for $\D$-modules \cite[Chapter 1.5]{hotta2007d}:
$$
f_{+}(M):=\mathbb{R}f_{\ast}(\D_{Y\leftarrow X}\otimes^{\mathbb{L}}M), \;\;\tn{and}\;\; f^{\dagger}(N):=\D_{X \to Y}\otimes^{\mathbb{L}} f^{-1}N[d_{X}-d_{Y}],
$$

where $\D_{Y\leftarrow X}$ and $\D_{X \to Y}$ denote the corresponding transfer bimodules. These functors induce functors on the bounded derived categories of mixed Hodge modules \cite[Section 8.3.3(m7)]{hotta2007d}, denoted  as follows:
$$
f_{\ast}:\tn{D}^b\tn{MHM}(X)\to \tn{D}^b\tn{MHM}(Y),\;\;\tn{and}\;\; f^{!}:\tn{D}^b\tn{MHM}(Y)\to \tn{D}^b\tn{MHM}(X).
$$

Given $\mc{M}\in \tn{D}^b\tn{MHM}(X)$, we say that $\mc{M}$ is mixed of weight $\leq w$ (resp. $\geq w$) if $\tn{gr}^W_i(\mc{H}^j(\mc{M}))=0$ for $i>j+w$ (resp. $i<j+w$). We say $\mc{M}$ is pure of weight $w$ if it is mixed of weight $\leq w$ and $\geq w$.

\subsection{Local cohomology as a mixed Hodge module }\label{LCP} Let $X$ be a smooth complex variety and let $Z\subseteq X$ be a closed subvariety. We write $\mathbb{R}\mc{H}^0_{Z}(-)$ for the functor on $\tn{D}^b_{h}(\D_X)$ of sections with support in $Z$, whose cohomology functors $\mc{H}^i_{Z}(-)$ are the local cohomology functors with support in $Z$.

We set $U=X\setminus Z$ with open immersion $j:U\to X$. Given $M\in \tn{D}_h^b(\D_X)$, there is a distinguished triangle in $\tn{D}^b_h(\D_X)$ \cite[Proposition 1.7.1(i)]{hotta2007d}:
\begin{equation}\label{triangle}
\mathbb{R}\mc{H}^0_{Z}(M)\longrightarrow M \longrightarrow j_{+}j^{\dagger}(M)\overset{+1\;}\longrightarrow.
\end{equation}
If $M$ underlies $\mc{M}\in \tn{D}^b\tn{MHM}(X)$, then $j_+j^{\dagger}(M)$ underlies $j_{\ast}j^{!}(\mc{M})\in \tn{D}^b\tn{MHM}(X)$, so this triangle endows $\mathbb{R}\mc{H}^0_{Z}(\mc{M})$ with the structure of an object in $\tn{D}^b\tn{MHM}(X)$. In particular, if $\mc{M}\in \tn{MHM}(X)$ then we have an exact sequence of mixed Hodge modules:
$$
0\longrightarrow \mc{H}^0_{Z}(\mc{M})\longrightarrow \mc{M}\longrightarrow \mc{H}^0(j_{\ast}j^!(\mc{M}))\longrightarrow \mc{H}^1_{Z}(\mc{M})\longrightarrow 0,
$$
and isomorphisms $\mc{H}^{q}_{Z}(\mc{M})\cong \mc{H}^{q-1}(j_{\ast}j^!(\mc{M}))$ for $q\geq 2$. For more information, see \cite[Section B.3]{MP4}.

When $X=\C^N$ is an affine space, we identify all of the above sheaves with their global sections, and view everything as a module over the Weyl algebra $\D=\Gamma(\C^N,\D_{\C^N})$. For ease of notation throughout, we write
$$
\mathbb{R}\Gamma_Z(-):=\mathbb{R}\Gamma_Z(\C^N,-),\;\;\tn{and}\;\; H^j_Z(-):=H^j_Z(\C^N,-),
$$
where $\Gamma_Z=\Gamma\circ \mathcal{H}^0_Z$, and $H^j_Z=\mathbb{R}^j\Gamma_Z$ are the global local cohomology functors.

\subsection{$\GL$-equivariant $\D$-modules on $\C^{m\times n}$}\label{equivar} We let $\X=\C^{m\times n}$ be the space of $m\times n$ generic matrices, with $m\geq n$. This space is endowed with an action of the group $\GL=\GL_m(\C)\times \GL_n(\C)$ via row and column operations, and the orbits stratify $\X$ by matrix rank. 

All $\D$-modules considered in this work are objects in the category $\tn{mod}_{\GL}(\D)$ of $\GL$-equivariant holonomic $\D$-modules. The simple objects in this category are the modules
$$
D_0,\;D_1,\cdots,\; D_n,
$$
where $D_p=\mc{L}(\Z_p,\X)$ is the intersection homology $\D$-module associated to $\Z_p$.

Given $M\in \tn{mod}_{\GL}(\D)$ and $0\leq q\leq n$, the local cohomology modules $H^j_{\Z_q}(M)$ are also objects of $\tn{mod}_{\GL}(\D)$, and thus have composition factors among $D_0,\cdots , D_n$. When $m\neq n$, the category $\tn{mod}_{\GL}(\D)$ is semi-simple \cite[Theorem 5.4(b)]{categories}, so each local cohomology module decomposes as a $\D$-module into a direct sum of its simple composition factors. For instance, each local cohomology module in Example \ref{Ex2} is semi-simple.

On the other hand, for square matrices, the category $\tn{mod}_{\GL}(\D)$ is not semi-simple \cite[Theorem 5.4(a)]{categories}. We let $m=n$, and let $S=\C[x_{i,j}]_{1\leq i,j\leq n}$ denote the ring of polynomial functions on $\X$. The localization $S_{\det}$ of the polynomial ring at the $n\times n$ determinant $\det=\det(x_{i,j})$ is a holonomic $\D$-module, with composition series as follows \cite[Theorem 1.1]{raicu2016characters}:
\begin{equation}\label{filtdet}
0 \subsetneq S \subsetneq \langle \det^{-1}\rangle_{\D}\subsetneq \langle \det^{-2}\rangle_{\D}\subsetneq\cdots \subsetneq \langle \det^{-n}\rangle_{\D}=S_{\det},
\end{equation}
where $\langle \det^{-p} \rangle_{\D}$ is the $\D$-submodule of $S_{\det}$ generated by $\det^{-p}$, and $\langle \det^{-p}\rangle_{\D}/\langle \det^{-p+1}\rangle_{\D}\cong D_{n-p}$. Following \cite{lHorincz2018iterated}, we define $Q_n=S_{\det}$, and for $p=0,\cdots ,n-1$, we set
\begin{equation}\label{defQ}
Q_p=\frac{S_{\det}}{\langle \det^{p-n+1}\rangle_{\D}}.	
\end{equation}
The modules $Q_p$ constitute the indecomposable summands of local cohomology with determinantal support \cite[Theorem 1.6]{lHorincz2018iterated}. Indeed, let us denote by $\tn{add}(Q)$ the additive subcategory of $\tn{mod}_{\GL}(\D)$ consisting of modules that are isomorphic to a direct sum of the modules $Q_0,\cdots , Q_n$. By \cite[Theorem 1.6]{lHorincz2018iterated}, 
 if $q<p$, each local cohomology module of the form $H^j_{\Z_q}(D_p)$ belongs to $\tn{add}(Q)$. For instance, in Example \ref{Ex1} the three  local cohomology modules are isomorphic as $\D$-modules to $Q_2$, $Q_1$, and $Q_0$, respectively. 

\subsection{Subrepresentations of equivariant $\D$-modules}\label{SecGLRep} For an integer $N\geq 1$, the irreducible representations of the general linear group $\GL_N(\C)$ are in one-to-one correspondence with \defi{dominant weights} 
$$
\lam=(\lam_1\geq \lam_2\geq \cdots \geq \lam_N)\in \mathbb{Z}^N.
$$
We write $\mathbb{Z}^N_{\tn{dom}}$ for the set of dominant weights, and $\bS_{\lam}\C^N$ for the irreducible representation corresponding to a dominant weight $\lam$, where $\bS_{\lam}$ is a \defi{Schur functor}. For $b\in \mathbb{Z}$ and $a\geq 0$ we write $(b^a)=(b,\cdots ,b,0,\cdots,0)$ for the dominant weight with $b$ repeated $a$ times. For instance, $\bS_{(d)}\C^N=\Sym^d\C^N$ and $\bS_{(1^d)}\C^N=\bw^d \C^N$.

For $0\leq p\leq n$ the module $D_p$ decomposes into irreducible $\GL$-representations as follows \cite[Section 5]{raicu2016characters}:
\begin{equation}\label{chardp}
D_p=\bigoplus_{\lam\in W^p}\bS_{\lam(p)}\C^m\otimes \bS_{\lam}\C^n,
\end{equation}
where 
$$
\lam(p)=(\lam_1,\cdots , \lam_p, \underbrace{p-n,\cdots ,p-n}_{m-n},\lam_{p+1}+(m-n),\cdots,\lam_n+(m-n)),
$$
 and
$$
W^p:=\big\{\lam \in \mathbb{Z}^n_{\dom}:\lam_p\geq p-n,\;\lam_{p+1}\leq p-m\big\}.
$$
If $V$ is a subrepresentation of a $\GL$-equivariant holonomic $\D$-module, then it is a subrepresentation of a finite direct sum of the modules $D_0,\cdots ,D_n$. Thus $V$ has a $\GL$-decomposition of the following form:
$$
V=\bigoplus_{0\leq p\leq n}  \bigoplus_{\lam \in W^p} \big( \bS_{\lam(p)}\C^m \otimes \bS_{\lam}\C^n\big)^{\oplus b_{\lam}(V)}.
$$
We encode the equivariant structure of such a $V$ via a multiset of dominant weights
$$
\mathfrak{W}(V)=\left\{ (\lam, b_{\lam}(V)): \lam\in \mathbb{Z}^n_{\dom}\right\}.
$$
Since the sets $W^p$ are pairwise disjoint, and $\lam(p)$ is uniquely determined by $\lam$, the multiset $\mathfrak{W}(V)$ completely describes the equivariant structure of $V$. Weights of direct sums are described by disjoint unions:
$$
\mathfrak{W}(V_1\oplus V_2)=\mathfrak{W}(V_1)\sqcup \mathfrak{W}(V_2)=\left\{ (\lam, b_{\lam}^1+b_{\lam}^2): (\lam, b_{\lam}^1)\in \mathfrak{W}(V_1),\; (\lam, b_{\lam}^2)\in \mathfrak{W}(V_2)\right\}.
$$
When $V$ is multiplicity-free (i.e. $b_{\lam}(V)\leq 1$ for all $\lam$), we simply write $\mathfrak{W}(V)$ as a set of dominant weights. For example, $\mathfrak{W}(D_p)=W^p$ for $p=0,\cdots , n$. Since $Q_p$ has simple composition factors $D_0,\cdots , D_p$, each with multiplicity one, we have
\begin{equation}\label{characterofqp}
\mathfrak{W}(Q_p)=\mathfrak{W}(D_0\oplus \cdots \oplus D_p)=W^0\sqcup \cdots \sqcup W^p.
\end{equation}
In particular, $\mathfrak{W}(Q_n)=\mathbb{Z}^n_{\dom}$. 

\subsection{Hodge modules on $\C^{m\times n}$}\label{hodged} We write $\IC^H_{\Z_p}$ for the pure Hodge module associated to the trivial variation of Hodge structure on $\Z_p\setminus \Z_{p-1}$, which has weight $d_p=\dim \Z_p=p(m+n-p)$, and overlies the simple $\D$-module $D_p$. Given $k\in \mathbb{Z}$, the $k$-th Tate twist of $\IC^H_{\Z_p}$, written $\IC^H_{\Z_p}(k)$, is pure of weight $d_p-2k$. 

We restrict our attention to  the case of square matrices, the situation when our local cohomology modules of interest belong to $\tn{add}(\mc{Q})$ (see Section \ref{equivar}). We classify the possible mixed Hodge module structures on the $\D$-modules $Q_p$. Let $\Z=\Z_{n-1}$ denote the determinant hypersurface, and we write $U=\X\setminus \Z$ with open immersion $j:U\to \X$. We define $\mc{Q}^H_n:=j_{\ast}\mc{O}_U^H$, where $\mc{O}^H_U$ is the trivial pure Hodge module on $U$. Up to a Tate twist, $\mc{Q}^H_n$ is the unique mixed Hodge module that may overlie $Q_n$ \cite[Section 4.1]{perlmanraicu}.

The weight filtration $W_{\bullet}$ on $\mc{Q}_n^H$ is described as follows: if $w<n^2$ or $w>n^2+n$, then $\tn{gr}^W_w\mc{Q}_n^H=0$, and 
\begin{equation}\label{weightQn}
\tn{gr}^W_{n^2+n-p}\mc{Q}_n^H=\IC^H_{\Z_p}\left(-\binom{n-p+1}{2}\right),\;\; \tn{for\;\; $p=0,\cdots , n$.}	
\end{equation}
In other words, the copy of $D_p$ in $\mc{Q}_n^H$ underlies a pure Hodge module of weight $n^2+n-p$. Using (\ref{weightQn}) we define a mixed Hodge module structure on each $Q_p$ for $0\leq p\leq n-1$ as follows. Consider the exact sequence
\begin{equation}
0\longrightarrow W_{n^2+n-p-1}(\mc{Q}_n^H)\longrightarrow \mc{Q}_n^H \longrightarrow \mc{Q}_n^H/W_{n^2+n-p-1}(\mc{Q}_n^H)\longrightarrow 0.	
\end{equation}
By (\ref{defQ}) and (\ref{weightQn}), it follows that $\mc{Q}_n^H/W_{n^2+n-p-1}(\mc{Q}_n^H)$ is a mixed Hodge module overlying $Q_p$. We define
\begin{equation}\label{defQph}
\mc{Q}^H_p=\frac{\mc{Q}_n^H}{W_{n^2+n-p-1}(\mc{Q}_n^H)},\;\tn{for\; $p=0,\cdots,n-1$}.
\end{equation}

\begin{prop}\label{uniqueQp}
Up to a Tate twist, $\mc{Q}_p^H$ is the only mixed Hodge module overlying $Q_p$.	
\end{prop}

The proof of Proposition \ref{uniqueQp} is identical to the proof for the case $Q_n$ in \cite[Section 4.1]{perlmanraicu}, except that $Q_p$ for $p<n$ does not have full support, so \cite[Equation (2.11)]{perlmanraicu} cannot be used. However, this is remedied via the following lemma. We write $\delta^p=((p-n)^n)\in \mathfrak{W}(D_p)$.

\begin{lemma}\label{QpHlemma}
Let $1\leq p\leq n$ and consider a mixed Hodge module overlying $Q_p$ with Hodge filtration $F_{\bullet}$. Given $1\leq r\leq p-1$, if $\delta^r\in \mathfrak{W}(F_{l}(Q_p))$ for some $l\in \mathbb{Z}$, then $\delta^r+(1^{r+1})\in \mathfrak{W}(F_l(Q_p))$.	
\end{lemma}

\begin{proof}
Suppose that $\delta^r\in \mathfrak{W}(F_{l}(Q_p))$, and let $m$ be a nonzero element of the corresponding isotypic component of $Q_p$. We let $S_{(1^{r+1})}\cong \bw^{r+1}\C^n\otimes \bw^{r+1}\C^n$ denote the subspace of the polynomial ring $S$ spanned by the $(r+1)\times (r+1)$ minors of the generic matrix of variables. Since $F_l(Q_p)$ is an $S$-submodule of $Q_p$, it suffices to show that the subspace $N^r=S_{(1^{r+1})}\cdot m$ of $Q_p$ is nonzero. Suppose for contradiction that $N^r=0$. Since $S_{(1^{r+1})}$ is the space of defining equations of $\Z_{r}$, it follows that the $S$-submodule of $Q_p$ generated by $m$ has support contained in $\Z_r$. Thus, $H^0_{\Z_{r}}(Q_p)\neq 0$, which implies that $Q_p$ has a $\D$-submodule with support contained in $\Z_{r}$. Since $r\leq p-1$, this is impossible, as $D_p$ has support $\Z_p$ and is the socle of $Q_p$ \cite[Lemma 6.3]{lHorincz2018iterated}. 
\end{proof}


\section{The weight filtration}\label{SecWeight}
Let $\X=\C^{m\times n}$ be the space of $m\times n$ matrices, with $m\geq n$.  Theorem \ref{Othm} is a consequence of our main result, which addresses the mixed Hodge module structure on any local cohomology module of the form $H^j_{\Z_q}(\IC^H_{\Z_{p}})$. We write  $d_p=\dim \Z_p=p(m+n-p)$ and $c_p=\operatorname{codim}(\Z_p,\X)=(m-p)(n-p) $. 

The goal of this section is to prove the following.

\begin{theorem}\label{main}
Let $0\leq q< p \leq n \leq m$ and $j\geq 0$. 
\begin{enumerate}
\item If $m=n$, then	$H^j_{\Z_q}(\IC^H_{\Z_{p}})$ is a direct sum of mixed Hodge modules of the form $\mc{Q}^H_r((c_p+n-q-j)/2)$ for $0\leq r\leq q$.
\item If $m\neq n$, then $H^j_{\Z_q}(\IC^H_{\Z_{p}})$ is a direct sum of pure Hodge modules of the form $\IC^H_{\Z_r}((d_r-d_p+r-q-j)/2)$ for $0\leq r\leq q$.
\end{enumerate}

In particular, in either case, if $\mc{M}$ is a simple composition factor of $H^j_{\Z_q}(\IC^H_{\Z_{p}})$ with support equal to $\Z_r$, then $\mc{M}$ is isomorphic to $\IC^H_{\Z_r}((d_r-d_p+r-q-j)/2)$.
\end{theorem}

The multiplicities of the summands in the statement of Theorem \ref{main} are determined by the formulas in \cite[Theorem 3.1, Theorem 6.1]{lHorincz2018iterated}, the former of which is recalled below.

As $\Z_n\setminus \Z_{n-1}$ is dense, we have $\mc{O}_{\X}^H=\IC^H_{\Z_n}$, so Theorem \ref{Othm} is a consequence of the case $p=n$ of Theorem \ref{main}. Our strategy for proving Theorem \ref{main} is explained in the Introduction. We discuss how one may deduce the Hodge filtration on these local cohomology modules from Theorem \ref{main} and \cite{perlmanraicu} in Section \ref{SecHodgeFilt}.

We recall the formula for the simple composition factors in this general setting, which we express as a generating function with coefficients in the Grothendieck group of $\GL$-equivariant holonomic $\D$-modules. Given integers $0\leq q< p\leq n \leq m$ we have the identity \cite[Theorem 3.1]{lHorincz2018iterated}: 
\begin{equation}\label{locCohForm}
\sum_{j\geq 0} \big[H^j_{\Z_q}\left(D_p\right)\big]\cdot t^j=\sum_{r=0}^q \left[D_r\right]\cdot t^{(p-q)^2+(p-r)\cdot (m-n)}\cdot \binom{n-r}{p-r}_{t^2}\cdot \binom{p-r-1}{q-r}_{t^2},
\end{equation}
where $\binom{a}{b}_{t}$ is a \defi{Gaussian binomial coefficient}, defined as follows. For $a\geq b$ we write:
$$
\binom{a}{b}_t=\frac{(1-t^a)\cdot (1-t^{a-1})\cdots (1-t^{a-b+1}) }{(1-t^b)\cdot (1-t^{b-1})\cdots (1-t)},
$$
with the convention that $\binom{a}{b}_t=0$ if $a<b$, and $\binom{a}{0}_t=\binom{a}{a}_t=1$. Specializing to $p=n$ recovers  the formula \cite[Main Theorem]{raicu2016local} for local cohomology of the polynomial ring with support in a determinantal variety.

To prove Theorem \ref{main}, we require the following piece of information from (\ref{locCohForm}).

\begin{lemma}\label{parityLemma}
Let $0\leq q<p\leq n\leq m$ and $j\geq 0$. The module $D_0$ appears as a simple composition factor of $H^j_{\Z_q}(D_p)$ only if $j\equiv (p-q)^2+p(m-n)$ \tn{($\mod\;2$)}.	
\end{lemma}

\begin{proof}
By (\ref{locCohForm}) the smallest degree in which $D_0$ appears is $(p-q)^2+p(m-n)$. The result then follows immediately from the fact that the Gaussian binomial coefficients in (\ref{locCohForm}) are supported in even degrees.	
\end{proof}

\subsection{The inductive structure of determinantal varieties}\label{SecInd} We now establish the inductive setup. Our treatment follows \cite[Section 2H]{lHorincz2018iterated} (see also \cite[Proposition 2.4]{bruns}).  

We choose coordinates $(x_{i,j})_{1\leq i\leq m, 1\leq j\leq n}$ on the space of matrices $\X$, and we let $\X_1$ denote the open subset of $\X$ defined by non-vanishing of the top left coordinate $x_{1,1}$. By performing row and column operations to eliminate entries in the first row and column of the generic matrix, we obtain an isomorphism:
$$
\X_1\cong \X'\times \C^{m-1}\times \C^{n-1}\times \C^{\ast},
$$
where $\X'$ is isomorphic to the space of matrices $\C^{(m-1)\times (n-1)}$, with coordinates
\begin{equation}\label{newcoord}
x_{i,j}'=x_{i,j}-\frac{x_{i,1}\cdot x_{1,j}}{x_{1,1}}.
\end{equation}
The copy of $\C^{\ast}$ above corresponds to the coordinate $x_{1,1}$, and the spaces $\C^{m-1}$ and $\C^{n-1}$ correspond to the remaining entries of the first column and row of $\X$, respectively.

For $p=0,\cdots ,n-1$, we write $\Z'_p\subseteq \X'$ for the determinantal variety of matrices of rank $\leq p$, with dimension $d'_p=p(m+n-p-2)$. We write $D_p'=\mc{L}(\Z_p',\X')$ for the intersection homology module associated to $\Z_p'$, and we write $Q_p'$ for the modules (\ref{defQ}) on $\X'$ in the case $m=n$.

Let $\phi:\X_1\to \X$ be the open immersion of $\X_1$ into $\X$, and let $\pi:\X_1 \to \X'$ be the projection map, noting that these are both smooth morphisms. For ease of notation we set $\X_1\cong \X'\times \Y$, so that $\pi$ has relative dimension $d_{\Y}=m+n-1$. We note that $d_p=d_p'+d_{\Y}$. 

Since $\Z_p$ is defined by the vanishing of the $(p+1)$-minors of the matrix of variables $(x_{i,j})$, one can show using (\ref{newcoord}) that $\phi^{-1}(\Z_p)=\pi^{-1}(\Z_{p-1}')$ for $1\leq p\leq n$. Further, if we write $\phi^{\ast}$, $\pi^{\ast}$ for the (non-shifted) inverse image functors of $\D$-modules, then we have the following isomorphisms: 
 \begin{equation}\label{pullD}
\phi^{\ast}(D_{p})\cong \pi^{\ast}(D_{p-1}'),\;\;\;\text{for\;\;$p=1,\cdots,n$}.
 \end{equation}
and we have
\begin{equation}\label{pullQ}
\phi^{\ast}(Q_{p})\cong \pi^{\ast}(Q_{p-1}'),\;\;\;\text{for\;\;$p=1,\cdots,n$}. 	
 \end{equation}

In addition, for all $1\leq q\leq p$, we have the following:
\begin{equation}\label{npullH}
\phi^{\ast}\big(H^j_{\Z_{q}}\big(D_p \big)\big)\cong \pi^{\ast}\big(H^j_{\Z'_{q-1}}\big(D'_{p-1} \big) \big),\;\;\;\text{for \;\;$j\geq 0$}.    
\end{equation}

\subsection{Hodge modules and the inductive setting}\label{SecIndHodge} In this subsection we determine how the mixed Hodge structure of local cohomology modules on $\X$ is related to that on $\X'$. The main result here is the following.

\begin{prop}\label{indcor}
Let $1\leq q<p \leq n\leq m$ and let $j\geq 0$. Let $U^0=\X\setminus \{0\}$ with open immersion $f:U^0\to \X$, and suppose that Theorem \ref{main} holds for the parameters $(q-1,p-1,n-1,m-1,j)$.
\begin{enumerate}
\item If $m=n$, then	$\mc{H}^0(f_{\ast}f^!H^j_{\Z_q}(\IC^H_{\Z_{p}}))$ is a direct sum of mixed Hodge modules of the form $\mc{Q}^H_r((c_p+n-q-j)/2)$ for $1\leq r\leq q$.
\item If $m\neq n$, then $\mc{H}^0(f_{\ast}f^!H^j_{\Z_q}(\IC^H_{\Z_{p}}))$ is a direct sum of pure Hodge modules of the form $\IC^H_{\Z_r}((d_r-d_p+r-q-j)/2)$ for $1\leq r\leq q$.
\end{enumerate}
\end{prop}

To prove Proposition \ref{indcor} above, we determine versions of the isomorphisms (\ref{pullD}), (\ref{pullQ}), and (\ref{npullH}) in the category of mixed Hodge modules, and then we follow the argument in \cite{lHorincz2018iterated} for the $\D$-module version of Proposition \ref{indcor}. The analogue of (\ref{npullH}) that we formulate must respect our choice of mixed Hodge structure on local cohomology (see Section \ref{LCP}). For this reason, we work in the derived categories of mixed Hodge modules on $\X$, $\X_1$, and $\X'$, and thus use the (cohomologically shifted) inverse image functors $\phi^{!}$ and $\pi^{!}$, which lift the $\D$-module functors $\phi^{\dagger}=\phi^{\ast}$ and $\pi^{\dagger}=\pi^{\ast}[d_{\Y}]$ respectively (see Section \ref{functors}).

Changing the functors in (\ref{pullD}), we immediately obtain the following isomorphisms of $\D$-modules:
\begin{equation}\label{pullopen}
\phi^{\dagger}(D_p)\cong \pi^{\dagger}\left(D_{p-1}'[-d_{\Y}]\right),\;\;\;\text{for\;\;$p=1,\cdots,n$}.	
\end{equation}

We start by translating these isomorphisms to the level of Hodge modules. 

\begin{lemma}\label{pullH}
Let $1\leq p\leq n$. We have the following isomorphims in $\tn{D}^b \tn{MHM}(\X_1)$ for all $k\in \mathbb{Z}$: 
\begin{equation}\label{pullH2}
\phi^{!}\left(\IC^H_{\Z_{p}}(k)\right)\cong \pi^{!}\left(\IC^H_{\Z_{p-1}'}(k)\left[-d_{\Y}\right]\right).
\end{equation}
\end{lemma}

\begin{proof}
It suffices to verify the case $k=0$. Both sides of (\ref{pullH2}) correspond to a variation of Hodge structure on the trivial local system on $(\Z_{p-1}'\setminus \Z_{p-2}')\times \Y$. Thus, we need to show that their weights match. By \cite[Theorem 8.3]{schnell}, given a smooth morphism $f:X\to Y$, and a pure Hodge module $\mc{N}$ of weight $v$ on $Y$, the inverse image $\mc{H}^{d_Y-d_X}(f^{!}(\mc{N}))$ is pure of weight $v+d_X-d_Y$ (see also \cite[2.26]{saito90}). Thus, $\phi^{!}(\IC^H_{\Z_p})$ is pure of weight $d_p$, and $\pi^{!}(\IC^H_{\Z'_{p-1}}[-d_{\Y}])$ is pure of weight $d_{p-1}'+d_{\Y}$. Since $d_p=d_{p-1}'+d_{\Y}$, the result follows. \end{proof}

\begin{prop}\label{reduction}
Let $1\leq q< p \leq n$. We have an isomorphism in $\tn{D}^b\tn{MHM}(\X_1)$:
$$
\phi^{!}\left(\mathbb{R}\Gamma_{\Z_{q}}\left(\IC^H_{\Z_{p}}\right)\right)\cong\pi^{!}\left(\mathbb{R}\Gamma_{\Z_{q-1}'}\left(\IC_{\Z_{p-1}'}^H\left[-d_{\Y}\right]\right)\right).
$$
\end{prop}

\begin{proof}
Let $U=\X\setminus \Z_{q}$, $U'=\X'\setminus \Z'_{q-1}$, and $U_1=\phi^{-1}(U)=\pi^{-1}(U')$, and consider the following commutative diagram:
\begin{equation}\label{diagram}
\begin{tikzcd}[column sep = huge]
U'\arrow[d, hook, "i"] & U_1 \arrow[l,  "\tilde{\pi}"']\arrow[r, "\tilde{\phi}"]\arrow[d, hook, "j"] & U \arrow[d, hook, "k"]\\
\X' & \X_1 \arrow[l, "\pi"]\arrow[r, hook, "\phi"'] & \X	
\end{tikzcd}
\end{equation}
where the vertical arrows are the open immersions, and $\tilde{\pi}$, $\tilde{\phi}$ are the maps on $U_1$ induced by $\pi$, $\phi$. Using the distinguished triangle (\ref{triangle}) and Lemma \ref{pullH}, it suffices to show that 
\begin{equation}\label{desiredisoH}
\phi^{!}k_{\ast}k^{!}\left(\IC^H_{\Z_{p}}\right)\cong \pi^{!}i_{\ast}i^{!}\left(\IC_{\Z_{p-1}'}^H\left[-d_{\Y}\right]\right).
\end{equation}
Since $U_1=\phi^{-1}(U)=\pi^{-1}(U')$, the base change theorem for Hodge modules \cite[4.4.3]{saito90} implies that $\phi^{!}k_{\ast}=j_{\ast}\tilde{\phi}^{!}$ and  $\pi^{!}i_{\ast}=j_{\ast}\tilde{\pi}^{!}$. Commutativity of (\ref{diagram}) implies $\tilde{\phi}^{!}k^{!}=j^{!}\phi^{!}$ and $\tilde{\pi}^{!}i^{!}=j^{!}\pi^{!}$, so we obtain
$$
\phi^{!}k_{\ast}k^{!}\left(\IC^H_{\Z_{p}}\right)\cong j_{\ast}j^{!}\phi^{!}\left(\IC^H_{\Z_{p}}\right),\;\;\tn{and}\;\; \pi^{!}i_{\ast}i^{!}\left(\IC_{\Z_{p-1}'}^H\left[-d_{\Y}\right]\right)\cong j_{\ast}j^{!}\pi^{!}\left(\IC_{\Z_{p-1}'}^H\left[-d_{\Y}\right]\right).
$$
The desired isomorphism (\ref{desiredisoH}) then follows from (\ref{pullH2}).
\end{proof}

As a corollary, we translate (\ref{npullH}) to the level of mixed Hodge modules.

\begin{corollary}\label{newcor}
Let $1\leq q< p \leq n$ and $j\geq 0$. We have an isomorphism in $\tn{D}^b\tn{MHM}(\X_1)$:
\begin{equation}\label{ffff}
\phi^{!}\left(H^j_{\Z_q}\left(\IC^H_{\Z_p} \right) \right)\cong \pi^{!}\left(H^j_{\Z_{q-1}'}\left( \IC^H_{\Z_{p-1}'} \right)\left[-d_{\Y}\right] \right).
\end{equation}
\end{corollary}

\begin{proof}
Since $\phi$ is smooth of relative dimension zero,  we have 
$$
\mc{H}^j\left( \phi^{!}\left(\mathbb{R}\Gamma_{\Z_{q}}\left(\IC^H_{\Z_{p}}\right)\right)\right)=\phi^{!}\left(H^j_{\Z_q}\left(\IC^H_{\Z_p} \right) \right).
$$	
On the other hand, since $\pi$ is smooth of relative dimension $d_{\Y}$, we have $$
\mc{H}^j\left(\pi^{!}\left(\mathbb{R}\Gamma_{\Z_{q-1}'}\left(\IC_{\Z_{p-1}'}^H\left[-d_{\Y}\right]\right)\right) \right) = \pi^{!}\left(H^j_{\Z_{q-1}'}\left(\IC^H_{\Z_{p-1}'} \right)\left[-d_{\Y}\right] \right).
$$
Thus, (\ref{ffff}) is the consequence of taking cohomology of the isomorphism stated in Proposition \ref{reduction}.
\end{proof}

For the remainder of the section, we denote by $\mc{Q}^{H'}_p$ the mixed Hodge module overlying $Q_p'$ defined in (\ref{defQph}).

\begin{corollary}\label{pushpullQ}
Let $1\leq r\leq  n=m$. We have the following isomorphims in $\tn{D}^b \tn{MHM}(\X_1)$ for all $k\in \mathbb{Z}$: 
$$
\phi^!\left(\mc{Q}_r^H(k)\right)\cong \pi^!\left(\mc{Q}_{r-1}^{H'}(k)\left[-d_T\right]\right).$$

\end{corollary}

\begin{proof}
By smoothness of $\phi$ and $\pi$, it suffices to prove the result after applying $\mc{H}^0$. We proceed by descending induction on $r$. By (\ref{desiredisoH}) for $p=n$ and $q=n-1$, we deduce the base case $r=n$, as $k_{\ast}k^{!}(\mc{O}^H_{\X})\cong \mc{Q}^H_n$ and $i_{\ast}i^{!}(\mc{O}^H_{\X'})=\mc{Q}^{H'}_{n-1}$. For the inductive step, we consider the diagram
$$
\begin{tikzcd}[column sep = small]
0 \arrow[r] & \mc{H}^0\left(\phi^{!}\left(\IC^H_{\Z_{r}}\left(-\binom{n-r+1}{2}\right)\right)\right)\arrow[d, "\sim"] \arrow[r] & \mc{H}^0\left(\phi^!\left(\mc{Q}_r^H\right) \right)\arrow[r]\arrow[d, "\sim"] & \mc{H}^0\left(\phi^!\left(\mc{Q}_{r-1}^H\right) \right)\arrow[r] & 0\\
0 \arrow[r] & \mc{H}^0\left(\pi^{!}\left(\IC^H_{\Z_{r-1}'}\left(-\binom{n-r+1}{2}\right)\left[-d_T\right]\right)\right)\arrow[r] & \mc{H}^0\left(\pi^!\left(\mc{Q}_{r-1}^{H'}\left[-d_T\right]\right)\right)\arrow[r] & \mc{H}^0\left(\pi^!\left(\mc{Q}_{r-2}^{H'}\left[-d_T\right]\right)\right)\arrow[r]	&  0
\end{tikzcd}
$$
The rows in the diagram arise from the definition (\ref{defQph}) of $\mc{Q}^H_r$, and smoothness of $\phi$ and $\pi$. The vertical isomorphisms follow from Lemma \ref{pullH} and the inductive hypothesis, respectively. 

Since morphisms of mixed Hodge modules are strict with respect to the Hodge and weight filtrations \cite[Section 8.3.3(m4)]{hotta2007d}, the rows of this diagram remain exact after applying $F_k$ or $W_k$ for some $k\in \mathbb{Z}$. As the left square is a commutative diagram on the underlying doubly-filtered $\D$-modules, there is an induced morphism from the doubly-filtered $\D$-module underlying $\mc{H}^0\left(\phi^!\left(\mc{Q}_{r-1}^H\right) \right)$ to that of $\mc{H}^0\left(\pi^!\left(\mc{Q}_{r-2}^{H'}\left[-d_T\right]\right)\right)$, and this map underlies a morphism of mixed Hodge modules (see \cite[Section 8.3.3(m1)]{hotta2007d}). By the Snake Lemma, we obtain the desired isomorphism.
\end{proof}

Now that we have versions of the isomorphisms (\ref{pullD}), (\ref{pullQ}), and (\ref{npullH}) in the category of mixed Hodge modules, we proceed with the proof of Proposition \ref{indcor}, following \cite[Section 6.2]{lHorincz2018iterated}. As in the statement of Proposition \ref{indcor}, we let $U^0=\X\setminus \{0\}$ with open immersion $f:U^0\to \X$. We write $f_1$ for the open immersion of $\X_1$ into $U^0$, so  $\phi=f\circ f_1$. The following is a Hodge-theoretic version of \cite[Lemma 6.7]{lHorincz2018iterated}. 

\begin{lemma}\label{newRest}
If $\mc{M},\mc{N}\in \operatorname{MHM}(U^0)$ lie over modules in 	$\tn{mod}_{\GL}(\D_{U^0})$, and $f_1^!(\mc{M})\cong f_1^{!}(\mc{N})$, then $\mc{M}\cong \mc{N}$.
\end{lemma}

\begin{proof}
This follows from the argument in the proof of \cite[Lemma 6.7]{lHorincz2018iterated}, replacing $\D$-module functors with the corresponding Hodge module functors (see also the last paragraph in the proof of Corollary \ref{pushpullQ}).
\end{proof}

\begin{lemma}\label{pushpulllemm}
For all $1\leq r\leq n$ we have
\begin{enumerate}
\item if $m=n$, then $\mc{H}^0(f_{\ast}f^! \mc{Q}_r^H)\cong \mc{Q}^H_r$,
\item if $m\neq n$, then $\mc{H}^0(f_{\ast}f^! IC^H_{\Z_r})\cong IC^H_{\Z_r}$.
\end{enumerate}
	
\end{lemma}

\begin{proof}
Let $\mc{M}$ be $\mc{Q}^H_r$ or 	$IC^H_{\Z_r}$. Since $r\geq 1$ we have $H^0_{\{0\}}(\mc{M})=0$. Thus, by (\ref{triangle}) for $j=f$, it suffices to show that $H^1_{\{0\}}(\mc{M})=0$. When $\mc{M}=\mc{Q}^H_r$, this is verified in the second paragraph of the proof of \cite[Proposition 6.8]{lHorincz2018iterated}, and when $m\neq n$ and $IC^H_{\Z_r}$, this follows from \cite[Theorem 3.1]{lHorincz2018iterated}.
\end{proof}

Finally, we prove Proposition \ref{indcor}. We write $c'_r=\operatorname{codim}(\Z_r',\X')=((m-1)-r)((n-1)-r)$.

\begin{proof}[Proof of Proposition \ref{indcor}]
Let $1\leq q<p\leq n\leq m$ and let $j\geq 0$. The hypothesis is:  $H^j_{\Z'_{q-1}}(IC^H_{\Z'_{p-1}})$ is a direct sum of Hodge modules of the form $\mc{Q}^{H'}_{r-1}((c_{p-1}'+n-q-j)/2)$ if $m=n$, and it is a direct sum of Hodge modules of the form $IC^H_{\Z_{r-1}'}((d'_{r-1}-d'_{p-1}+r-q-j)/2)$ if $m\neq n$.  

Assume first that $m=n$. We want to show that $\mc{H}^0(f_{\ast}f^!H^j_{\Z_q}(\IC^H_{\Z_{p}}))$ is a direct sum of Hodge modules of the form $\mc{Q}^H_r((c_p+n-q-j)/2)$ for $1\leq r\leq q$. By Lemma \ref{pushpulllemm}(1) it suffices to show that $f^!H^j_{\Z_q}(\IC^H_{\Z_{p}})$ is a direct sum of Hodge modules of the form $f^{!}\mc{Q}^H_r((c_p+n-q-j)/2)$.  The argument is identical to the second-to-last paragraph of the proof of \cite[Proposition 6.8]{lHorincz2018iterated}: by Lemma \ref{newRest} it suffices to prove the result after restricting to $\X_1$. By Corollary \ref{newcor} and our hypothesis there is a  sequence of isomorphisms
$$
f_1^{!}f^!H^j_{\Z_q}(\IC^H_{\Z_{p}})=\phi^{!}H^j_{\Z_q}(\IC^H_{\Z_{p}})\cong \pi^{!}\left(H^j_{\Z_{q-1}'}( \IC^H_{\Z_{p-1}'})\left[-d_{\Y}\right] \right)\cong\bigoplus_{1\leq r\leq q} \pi^!\left(\mc{Q}_{r-1}^{H'}((c_{p-1}'+n-q-j)/2)\left[-d_T\right]\right).
$$
Since $c'_{p-1}=((n-1)-(p-1))^2=(n-p)^2=c_p$, we obtain the desired result by Corollary \ref{pushpullQ}.

The proof for the case $m\neq n$ is identical, and uses Lemma \ref{pullH} and Lemma \ref{pushpulllemm}(2).
\end{proof}

\subsection{Proof of Theorem \ref{main}}\label{SecProof}  In this section we complete the proof of Theorem \ref{main}. Let $\X=\C^{m\times n}$ be the space of $m\times n$ generic matrices, with $m\geq n$. We write $\Z_p\subseteq \X$ for the determinantal variety of matrices of rank $\leq p$, with $d_p=\dim \Z_p=p(m+n-p)$. 

We proceed by induction on $n\geq 1$. The following implies the case $n=1$, and will serve as base case when we later induce on $q\geq 0$.

\begin{lemma}\label{q0}
For $j\geq 0$ we have that $H^j_{\Z_0}(\IC^H_{\Z_{p}})$ is a direct sum of copies of $\IC^H_{\Z_0}((-d_p-j)/2)$.
\end{lemma}

\begin{proof}
Let $i:\{0\}\hookrightarrow \X$ be the closed immersion of the origin ($=\Z_0$). By \cite[Proposition 1.7.1(iii)]{hotta2007d} we have that $i_{\ast}i^!=\mathbb{R}\Gamma_{\{0\}}(-)$ on the category $\tn{D}^b\tn{MHM}(\X)$, in a manner compatible with our fixed choice of mixed Hodge module structure on local cohomology. Since $i$ is the inclusion of the origin, Kashiwara's equivalence and \cite[Section 8.3.3(m13)]{hotta2007d} imply that the functor $i_{\ast}$ preserves pure complexes and their weight. Thus, it suffices to show that for all $j\in \mathbb{Z}$ we have that $\mc{H}^j(i^!(\IC^H_{\Z_p}))$ is a direct sum of copies of $\mathbb{Q}^H((-d_p-j)/2)$, where $\mathbb{Q}^H(k)$ denote the Hodge structures of Tate \cite[(I-3)]{peterssteen}.

The determinantal variety $\Z_p$ admits a small resolution of singularities $\pi:Y_p\to \Z_p$, where $Y_p$ is the total space of a vector bundle on the Grassmannian $\operatorname{Gr}(p,\C^n)$, and $\pi^{-1}(\{0\})\cong \operatorname{Gr}(p,\C^n)$ \cite[Section 3.3]{MR617466}. By base change for mixed Hodge modules \cite[4.4.3]{saito90} we have $i^!(\pi_{\ast}\mc{O}^H_{Y_p})\cong \tilde{\pi}_{\ast}(\tilde{i}^!\mc{O}^H_{Y_p})$,
where $\tilde{i}$ is the inclusion of $\pi^{-1}(\{0\})$ into $Y_p$ and $\tilde{\pi}$ is the restriction of $\pi$ to $\pi^{-1}(\{0\})$. Since $\pi$ is small we have $\pi_{\ast}\mc{O}^H_{Y_p}\cong \IC^H_{\Z_p}$ \cite[Section 8.3.3(p7), Proposition 8.2.30]{hotta2007d}. On the other hand, since $\pi^{-1}(\{0\})\cong \operatorname{Gr}(p,\C^n)$ is smooth and $\tilde{i}$ is a closed immersion of codimension $pm$, we have $\tilde{i}^!\mc{O}^H_{Y_p}\cong \mc{O}^H_{\operatorname{Gr}(p,\C^n)}[-pm](-pm)$. Next, for all $j\in \mathbb{Z}$ we have isomorphisms of mixed Hodge structures $\mc{H}^j(\tilde{\pi}_{\ast}\mc{O}^H_{\operatorname{Gr}(p,\C^n)})\cong H^{j+p(n-p)}(\operatorname{Gr}(p,\C^n))$, where $H^{k}(\operatorname{Gr}(p,\C^n))$ denotes the $k$-th de Rham cohomology of $\operatorname{Gr}(p,\C^n)$, endowed with its standard pure Hodge structure of weight $k$ (see Section 8.3.3(p7) and pg. 225 of \cite{hotta2007d}).

Putting it all together, we conclude that for all $j\in \mathbb{Z}$ we have the following isomorphism of mixed Hodge structures
$$
 \mc{H}^j(i^!(\IC^H_{\Z_p}))\cong H^{j-p(m-n+p)}(\operatorname{Gr}(p,\C^n))(-pm).
$$
Since $\operatorname{Gr}(p,\C^n)$ admits an affine paving, it follows that $H^{2k}(\operatorname{Gr}(p,\C^n))$ is a direct sum of copies of $\mathbb{Q}^H(-k)$ for all $0\leq k\leq p(n-p)$ (see \cite[Proposition 5.7.5]{Achar}). Therefore, $ \mc{H}^j(i^!(\IC^H_{\Z_p}))$ is a direct sum of copies of $\mathbb{Q}^H((p(m-n+p)-j)/2-pm)=\mathbb{Q}^H((-d_p-j)/2)$, as claimed. \end{proof}

As the base case is complete, we assume $n\geq 2$, and that Theorem \ref{main} holds for $n'<n$. As in the statement of Proposition \ref{indcor}, we let $U^0=\X\setminus \{0\}$ with open immersion $f:U^0\to \X$. For all $j\geq 0$ we have $H^1_{\{0\}}(H^j_{\Z_q}(IC^H_{\Z_p}))=0 $ \cite[Theorem 1.3, Theorem 1.5]{lHorincz2018iterated}, so by (\ref{triangle}) there is a short exact sequence 
\begin{equation}\label{impseq}
0\lra H^0_{\{0\}}(H^j_{\Z_q}(IC^H_{\Z_p}))\lra H^j_{\Z_q}(IC^H_{\Z_p})\lra \mc{H}^0(f_{\ast}f^{!}(H^j_{\Z_q}(IC^H_{\Z_p}))\lra 0.
\end{equation}
We summarize what we may conclude from the inductive hypothesis on $n$, and the inductive structure of determinantal varieties.

\begin{claim}\label{hypoth}
Assuming the inductive hypothesis on $n$, we have the following for all $j\geq 0$:
\begin{enumerate}
\item The module $\mc{H}^0(f_{\ast}f^{!}(H^j_{\Z_q}(IC^H_{\Z_p}))$ is of the desired form: for $m=n$ it is a direct sum of modules of the form $\mc{Q}^H_r((c_p+n-q-j)/2)$ for $1\leq r\leq q$, and for $m\neq n$ it is a direct sum of modules of the form $\IC^H_{\Z_r}((d_r-d_p+r-q-j)/2)$ for $1\leq r\leq q$. In both cases, all composition factors with support equal to $\Z_r$ are isomorphic to $\IC^H_{\Z_r}((d_r-d_p+r-q-j)/2)$.

\item In order to prove Theorem \ref{main}, it suffices to show that $H^0_{\{0\}}(H^j_{\Z_q}(IC^H_{\Z_p}))$ is a direct sum of copies of $\IC^H_{\Z_0}((-d_p-q-j)/2)$.	
\end{enumerate}

\end{claim}

\begin{proof}
Item (1) is an immediate consequence of Proposition \ref{indcor} and the inductive hypothesis. To prove (2), we note that by (1) the module $\mc{H}^0(f_{\ast}f^{!}(H^j_{\Z_q}(IC^H_{\Z_p}))$ is mixed of weight $\leq d_p+q+j$. If $H^0_{\{0\}}(H^j_{\Z_q}(IC^H_{\Z_p}))$ is a direct sum of copies of $\IC^H_{\Z_0}((-d_p-q-j)/2)$, it is pure of weight $d_p+q+j$. By \cite[Corollary 1.10]{saito89}, the  sequence (\ref{impseq}) must split, yielding the desired result. 
\end{proof}

We now proceed by induction on $q\geq 0$, the base case being Lemma \ref{q0}. Going forward, we fix $1\leq q<p\leq n\leq m$ and $j\geq 0$. The inductive hypothesis on $q$ will not be invoked until the proof of Claim \ref{finalCl}.

Let $\mc{M}_0$ be a composition factor of $H^0_{\{0\}}(H^j_{\Z_q}(IC^H_{\Z_p}))$. By Claim \ref{hypoth}(2) and \cite[Corollary 1.10]{saito89}, to prove Theorem \ref{main}, it suffices to verify the following.

\begin{goal}\label{goal}
The pure Hodge module $\mc{M}_0$ is isomorphic to $\IC^H_{\Z_0}((-d_p-q-j)/2)$.		
\end{goal}

To this end, we examine weights in the following Grothendieck spectral sequence:
\begin{equation}\label{GSS}
E_2^{s,t}=H^s_{\Z_{q-1}}\big(H^t_{\Z_q}\big(\IC^H_{\Z_p}\big)\big)\implies H^{s+t}_{\Z_{q-1}}\big( \IC^H_{\Z_p}\big).
\end{equation}
For $u\geq 2$, the differentials on the $u$-th page are written
\begin{equation}\label{diffsE}
d^{s,t}_u:E_u^{s,t}\longrightarrow E_u^{s+u,t-u+1}.	
\end{equation}

We will use this spectral sequence and the inductive hypotheses to prove the existence of a nonzero morphism from $\mc{M}_0$ to a copy of $\IC^H_{\Z_0}((-d_p-q-j)/2)$.	

Since $\mc{M}_0$ is supported on the origin, we have $H^0_{\Z_{q-1}}(\mc{M}_0)=\mc{M}_0$, and $H^s_{\Z_{q-1}}(\mc{M}_0)=0$ for $s\geq 1$. As $H^0_{\{0\}}(H^j_{\Z_q}(IC^H_{\Z_p}))$ is a submodule of $H^j_{\Z_q}(IC^H_{\Z_p})$, we may identify $\mc{M}_0$ with a composition factor of $E_2^{0,j}$.

By (\ref{GSS}) we have $E^{s,t}_2=0$ for $s<0$ and for $t<0$, so it follows from (\ref{diffsE}) that there are no nonzero differentials to $E_u^{0,j}$ for all $u\geq 2$. Thus, $\mc{M}_0$ is a composition factor of $E^{0,j}_{\infty}$ unless $\mc{M}_0$ supports a nonzero differential on some page of the spectral sequence.

\begin{claim}\label{spectrallem}
For some $v\geq 2$ there is a nonzero morphism from $\mc{M}_0$ to $E_v^{v,j-v+1}$.
\end{claim}

\begin{proof}
Since $\mc{M}_0$ is supported on the origin, Lemma \ref{parityLemma} implies that $j\equiv (p-q)^2+p(m-n)$ ($\mod\;2$). Suppose for contradiction that the claim is false. By the discussion above, it follows that $\mc{M}_0$ is a composition factor of $E^{0,j}_{\infty}$, so $\mc{M}_0$ is a composition factor of $H^j_{\Z_{q-1}}(\IC^H_{\Z_p})$. Again using Lemma \ref{parityLemma}, we have $j\equiv (p-q+1)^2+p(m-n)$ ($\mod\; 2$). However, 
$$
(p-q)^2+p(m-n)-((p-q+1)^2+p(m-n))=-2p+2q-1,
$$
which is odd, yielding a contradiction.
\end{proof}

 Going forward, we write $v$ for the $v$ that satisfies Claim \ref{spectrallem}. It will turn out that the precise value of $v$ is irrelevant for our purposes. 
 Using the inductive hypothesis, we will show that all composition factors of $E_v^{v,j-v+1}$ supported on the origin are isomorphic to $\IC^H_{\Z_0}((-d_p-q-j)/2)$. To do so, it suffices to prove the same about $E_2^{v,j-v+1}$, of which $E_v^{v,j-v+1}$ is a subquotient.

\begin{claim}\label{techlem}
The mixed Hodge module $E_2^{v,j-v+1}$ is isomorphic to $H^v_{\Z_{q-1}}(\mc{N})$, where $\mc{N}$ is a direct sum of copies of $\IC^H_{\Z_q}((d_q-d_p+v-j-1)/2)$.	
\end{claim}

\begin{proof}
For ease of notation we set $\mathcal{M}=H^{j-v+1}_{\Z_q}(\IC^H_{\Z_p})$, and we recall that $E_2^{v,j-v+1}:=H^v_{\Z_{q-1}}(\mathcal{M})$. As discussed prior to the statement of Claim \ref{hypoth}, there is a short exact sequence
$$
0\lra H^0_{\{0\}}(\mathcal{M})\lra \mathcal{M}\lra \mc{H}^0(f_{\ast}f^{!}(\mathcal{M}))\lra 0,
$$
which induces a long exact sequence of local cohomology with support in $\Z_{q-1}$. Since $H^0_{\{0\}}(\mathcal{M})$ is supported on the origin, we have $H^s_{\Z_{q-1}}(H^0_{\{0\}}(\mathcal{M}))=0$ for all $s\geq 1$, so that (since $v\geq 2$) we have
$$
E_2^{v,j-v+1}=H^v_{\Z_{q-1}}(\mathcal{M})\cong H^v_{\Z_{q-1}}(	\mc{H}^0(f_{\ast}f^{!}(\mathcal{M}))).
$$
If $m\neq n$, then Claim \ref{hypoth}(1) implies that $\mc{H}^0(f_{\ast}f^{!}(\mathcal{M}))$ is a direct sum of Hodge modules of the form $\IC^H_{\Z_i}((d_i-d_p+i-q+v-j-1)/2)$ for $1\leq i\leq q$. For $i\leq q-1$, each $\IC^H_{\Z_i}$ is supported in $\Z_{q-1}$, so we have that $H^0_{\Z_{q-1}}(\IC^H_{\Z_i})=\IC^H_{\Z_i}$ and $H^s_{\Z_{q-1}}(\IC^H_{\Z_i})=0$ for $s\geq 1$. Therefore, since $v\geq 2$, it follows that $H^v_{\Z_{q-1}}(	\mc{H}^0(f_{\ast}f^{!}(\mathcal{M})))$ and thus $E_2^{v,j-v+1}$ is a direct sum of copies of $H^v_{\Z_{q-1}}(\IC^H_{\Z_q}((d_q-d_p+v-j-1)/2))$.

If $m=n$, then Claim \ref{hypoth}(1) implies that $\mc{H}^0(f_{\ast}f^{!}(\mathcal{M}))$ is a direct sum of Hodge modules of the form $\mc{Q}^H_i((c_p+n-q+v-j-1)/2)$ for $1\leq i\leq q$. For $i\leq q-1$, each $\mc{Q}^H_i$ is supported in $\Z_{q-1}$, so we have $H^0_{\Z_{q-1}}(\mc{Q}^H_i)=\mc{Q}^H_i$ and $H^s_{\Z_{q-1}}(\mc{Q}^H_i)=0$ for $s\geq 1$. Therefore, since $v\geq 2$, it follows that $H^v_{\Z_{q-1}}(	\mc{H}^0(f_{\ast}f^{!}(\mathcal{M})))$ and thus $E_2^{v,j-v+1}$ is a direct sum of copies of $H^v_{\Z_{q-1}}(\mc{Q}^H_q((c_p+n-q+v-j-1)/2))$. The short exact sequence
$$
0 \lra \IC^H_{\Z_q}((d_q-d_p+v-j-1)/2) \lra \mc{Q}^H_q((c_p+n-q+v-j-1)/2)\lra \mc{Q}^H_{q-1}((c_p+n-q+v-j-1)/2) \lra 0,
$$
induces a long exact sequence of local cohomology with support in $\Z_{q-1}$. Since $v\geq 2$, we conclude that $H^v_{\Z_{q-1}}(\mc{Q}^H_q((c_p+n-q+v-j-1)/2))$ is isomorphic to $H^v_{\Z_{q-1}}(\IC^H_{\Z_q}((d_q-d_p+v-j-1)/2))$. Thus, $E_2^{v,j-v+1}$ is isomorphic to a direct sum of copies of $H^v_{\Z_{q-1}}(\IC^H_{\Z_q}((d_q-d_p+v-j-1)/2))$. 
\end{proof}

\begin{claim}\label{finalCl}
All composition factors of $E_2^{v,j-v+1}$ supported on the origin are isomorphic to the pure Hodge module $\IC^H_{\Z_0}((-d_p-q-j)/2)$.
\end{claim}

\begin{proof}
By Claim \ref{techlem}, we have that $E_2^{v,j-v+1}$ is equal to $H^v_{\Z_{q-1}}(\mc{N})$, where $\mc{N}$ is a direct sum of copies of $\IC^H_{\Z_q}((d_q-d_p+v-j-1)/2)$. Since Tate twists commute with local cohomology functors, it follows from inductive hypothesis on $q$ that a composition factor of $E_2^{v,j-v+1}$ supported on the origin is of the form  
$$
\IC^H_{\Z_0}\left((-d_q-(q-1)-v)/2\right)\left(d_q-d_p+v-j-1)/2\right)=\IC^H_{\Z_0}((-d_p-q-j)/2),
$$ 
as required. 
\end{proof}

\begin{proof}[Conclusion of proof of Theorem \ref{main}] By Claim \ref{spectrallem} there is a nonzero morphism from $\mc{M}_0$ to $E_v^{v,j-v+1}$. Because $\mc{M}_0$ is supported on the origin, its image is supported on the origin, and is a Hodge submodule of $E_v^{v,j-v+1}$. Claim \ref{finalCl} implies that the image of $\mc{M}_0$ is isomorphic to $\IC^H_{\Z_0}((-d_p-q-j)/2)$, and since $\mc{M}_0$ is simple it follows that it is isomorphic to $\IC^H_{\Z_0}((-d_p-q-j)/2)$, as required to complete Goal \ref{goal}, and thus the proof of Theorem \ref{main}.	
\end{proof}


\section{The Hodge filtration}\label{SecHodgeFilt}

In this section we explain how to deduce the Hodge filtration on local cohomology from Theorem \ref{main} and the previous work \cite{perlmanraicu} regarding the Hodge filtration on the intersection cohomology pure Hodge modules. As an application, we determine the generation level and carry out another example. 
\subsection{The Hodge filtrations on the equivariant pure Hodge modules}\label{SecHodgeDp} Let $0\leq p \leq n$ and $k\in \mathbb{Z}$, and let $F_{\bullet}$ denote the Hodge filtration on $\IC^H_{\Z_p}(k)$. The simple $\D$-module $D_p$ underlying $\IC^H_{\Z_p}(k)$ is a representation of the group $\GL=\GL_m(\C)\times \GL_n(\C)$, and the filtered pieces $F_{\bullet}(\IC^H_{\Z_p}(k))$ are $\GL$-subrepresentations of $D_p$. We write $W^p=\mathfrak{W}(D_p)$ and $c_p=\codim \Z_p=(m-p)(n-p)$ (see Section \ref{SecGLRep} for the definition of $\mathfrak{W}(-)$) . For $l\in \mathbb{Z}$,  the dominant weights of the pieces of $F_{\bullet}$ are given by \cite[Theorem 3.1]{perlmanraicu}:
\begin{equation}\label{transF}
\mathfrak{W}\big(F_l\big(\IC^H_{\Z_p}\left(k\right)\big)\big)=\fD^p_{l-c_p-k},\tn{\;\;where\;\;}\fD^p_d:=\big\{\lam \in W^p:\lam_{p+1}+\cdots +\lam_n\geq -d-c_p\big\}. \end{equation}
Thus, the filtration is determined by the sum of the last $n-p$ entries of $\lam \in W^p$. It follows that 
\begin{equation}\label{start}
F_l(\IC^H_{\Z_p}(k))=0\;\; \tn{for $l<c_p+k$, and}\;\; F_{c_p+k}(\IC^H_{\Z_p}(k))\neq 0. 
\end{equation}
One observes the general phenomena that: (1) Tate twisting by $k$ amounts to a shift of the filtration by $k$, and (2) the Hodge filtration on $\IC^H_{Z}(k)$ starts in level $\operatorname{codim}(Z)+k$ (see \cite[Lemma 2.1]{perlmanraicu}). 

\subsection{The Hodge filtration on a local cohomology module}\label{SecHodgeLoc} 

We now describe the equivariant structure of the Hodge filtration $F_{\bullet}$ on local cohomology with determinantal support. 
The following is an immediate consequence of Theorem \ref{main}, (\ref{transF}), and the fact that morphisms in the category of mixed Hodge modules are strict with respect to the Hodge filtration (see \cite[Section 8.3.3(m4)]{hotta2007d}).

\begin{theorem}\label{Hfilthm}
Let $0\leq q<p\leq n\leq m$ and $j\geq 0$. 
Let $a_r$ denote the multiplicity of $D_r$ as a composition factor of $H^j_{\Z_q}(D_p)$. The Hodge filtration is encoded by the following multiset of dominant weights
\begin{equation}\label{Whodge}
\mathfrak{W}\big(F_k\big( H^j_{\Z_q}(\IC^H_{\Z_p})\big) \big)=\bigsqcup_{r=0}^q \big(\fD_{k-(c_r+c_p+r-q-j)/2}^r\big)^{\sqcup a_r},\;\;\tn{for\; $k\in \mathbb{Z}$}.	
\end{equation}

\end{theorem}

The next two remarks elucidate the $\mc{O}_{\X}$-module structure of the Hodge filtration in some cases.

\begin{remark}\label{ClaudiuRemark}
After the appearance of an earlier version of this article, it was pointed to us by Claudiu Raicu	 that, in the case of non-square matrices, the Hodge filtration on the codimension local cohomology module has the following pleasing description in terms of elements annihilated by powers of the defining ideal. 

Let $0\leq q< n <m$, let $S=\Gamma (\X,\mc{O}_{\X})$, let $c=c_q$, and let $J=J_{q+1}$ be the determinantal ideal defining $\Z=\Z_q$. We have for all $k\geq 0$:
\begin{equation}\label{claudiuExt}
F_k\big(H^{c}_{\Z}\left(\mc{O}_{\X}^H\right)\big)=\big\{ h\in H^{c}_{\Z}\left(\mc{O}_{\X}\right) \mid J^{k+1}\cdot h=0\big\}.
\end{equation}
Indeed, by the proof of \cite[Proposition 3.1(i)]{huneke}, the right side of (\ref{claudiuExt}) is equal to $\operatorname{Ext}^c_S(S/J^{k+1},S)$, whose equivariant structure is calculated as a special case of \cite[Main Theorem]{regularity}. Comparing this to the equivariant description given in our Theorem \ref{Hfilthm} yields the desired result. 

For square matrices, the above argument can be used to show that, for all $0< q< n$, (\ref{claudiuExt}) is true for $k=0$, but is false for all $k\geq 1$. In both cases, it would be interesting to relate $\Ext$ to the Hodge filtration on the higher local cohomology modules in the spirit of \cite[Section C.7]{MP4}.
\end{remark}

\begin{remark}\label{HodgeIdealRemark}
One may express the Hodge filtration on each $\mc{Q}^H_p$ in terms of the Hodge ideals $I_k(\Z)$ \cite{MP1} for the determinant hypersurface $\Z=\Z_{n-1}$, which combined with Theorem \ref{main} provides an alternate expression for the Hodge filtration on local cohomology with determinantal support in square matrices. 

In our setting, the Hodge ideals $I_k(\Z)$ are defined to be the ideals that satisfy the following:
\begin{equation}\label{standard}
F_k(\mc{Q}^H_n)=I_k(\Z)\otimes \mc{O}_{\X}((k+1)\cdot \Z),\;\;\;\;k\geq 0,
\end{equation}
and they were calculated in \cite[Theorem 1.1]{perlmanraicu}.

Given $a,b\geq 0$ we write $I_{a\times b}$ for $\GL$-equivariant ideal in $S=\Gamma(\X,\mc{O}_{\X})$ generated by the irreducible subrepresentation $\bS_{(b^a)}\C^n\otimes \bS_{(b^a)}\C^n\subseteq S$ (see \cite{DEP}). The ideal $I_{a\times b}$ is the smallest $\GL$-equivariant ideal containing the $b$-th powers of the $a\times a$ minors of $(x_{i,j})$. For all $0\leq p\leq n=m$ and all $k\geq 0$ we have

\begin{equation}\label{hodgeqpideal}
F_k\big(\mc{Q}^H_p\big)=\left( \frac{I_k(\Z)}{ I_{(p+1)\times (k-(n-p)+2)}\cap I_k(\Z)}\right)\otimes \mc{O}_{\X}((k+1)\cdot \Z),
\end{equation}
\vspace{.5mm}

\noindent with the convention that $I_{a\times b}=0$ if $a>n$ and $I_{a\times b}=S$ if $b<0$. Indeed, (\ref{hodgeqpideal}) follows immediately from (\ref{standard}), (\ref{characterofqp}), and the equivariant description of $I_{a\times b}$ \cite[Theorem 4.1]{DEP}.

\end{remark}

\subsection{Generation level}\label{SecGenLev} Given a filtered $\D$-module $(M,F_{\bullet})$, we say that $F_{\bullet}$ is \defi{generated in level $q$} if
$$
F_{l}(\D) \cdot F_q(M)=F_{q+l}(M) \textnormal{\;\;for all\; $l\geq 0$},
$$
where $F_{\bullet}(\D)$ denotes the order filtration on $\D$. The \defi{generation level} of $(M,F_{\bullet})$ is defined to be the minimal $q$ such that $F_{\bullet}$ is generated in level $q$. We determine the generation level of each local cohomology module.

\begin{prop}
Let $0\leq q< p \leq n \leq m$ and $j\geq 0$. Let $s$ be minimal such that $D_{s}$ is a simple composition factor of $H^j_{\Z_q}(D_p)$. Then the generation level of the Hodge filtration on $H^j_{\Z_q}(\IC^H_{\Z_p})$ is $(c_{s}+c_p+s-q-j)/2$.

If $m=n$, then $s=0$, so the generation level is $(n^2+c_p-q-j)/2$.
\end{prop}

\begin{proof}
Let $g=(c_{s}+c_p+s-q-j)/2$. We first show that the generation level is at most $g$. It suffices to verify that the generation level of the induced Hodge filtration on each composition factor is at most $g$. By Theorem \ref{main} we are interested in $\IC^H_{\Z_r}((d_r-d_p+r-q-j)/2)$ for $0\leq r \leq q$. The Hodge filtration on $\IC^H_{\Z_r}((d_r-d_p+r-q-j)/2)$ has generation level $(c_{r}+c_p+{r}-q-j)/2$, the first nonzero level \cite[Section 4.2]{perlmanraicu}. Since $g\geq (c_{r}+c_p+{r}-q-j)/2$ for all $r\geq s$, we conclude that the generation level is at most $g$.

It remains to show that the generation level of the Hodge filtration is at least $g$. By (\ref{start}) and Theorem \ref{Hfilthm}, the induced Hodge filtration on $D_{s}$ starts in level $g$, so it suffices to show that $D_{s}$ is a quotient of $H^j_{\Z_q}(D_p)$. When $m\neq n$, $H^j_{\Z_q}(D_p)$ is a semi-simple $\D$-module, so $D_{s}$ is necessarily a quotient. If $m=n$, then $H^j_{\Z_q}(D_p)$ is an object of $\tn{add}(Q)$, so $D_0$ is a quotient and $s=0$. 
\end{proof}

\begin{example}\label{Ex3}
Let $m=5$, $n=3$, $p=2$, and $q=1$.	Using (\ref{locCohForm}) the nonzero local cohomology modules are expressed in the Grothendieck group of $\GL$-equivariant holonomic $\D$-modules as follows:
$$
\big[H^3_{\Z_1}( D_2)\big]=[D_1],\;\;\big[H^5_{\Z_1}( D_2)\big]=[D_1]+[D_0],\;\;\big[H^7_{\Z_1}( D_2)\big]=\big[H^9_{\Z_1}( D_2)\big]=[D_0].
$$
If $D_2$ underlies the pure Hodge module $\IC^H_{\Z_2}$, Theorem \ref{main} implies that the simple composition factors above underlie Hodge modules with the following weights (from left to right): $15$, $17$, $18$, $20$, $22$.
For $k\in \mathbb{Z}$ we have
$$
\mathfrak{W}\big(F_k\big(H^3_{\Z_1}\big( \IC^H_{\Z_2}\big)\big)\big)=\fD^1_{k-4},\;\;\;\mathfrak{W}\big(F_k\big(H^5_{\Z_1}\big( \IC^H_{\Z_2}\big)\big)\big)=\fD^1_{k-3}\sqcup \fD^0_{k-6},
$$
$$
\mathfrak{W}\big(F_k\big(H^7_{\Z_1}\big( \IC^H_{\Z_2}\big)\big)\big)=\fD^0_{k-5},\;\;\mathfrak{W}\big(F_k\big(H^9_{\Z_1}\big( \IC^H_{\Z_2}\big)\big)\big)=\fD^0_{k-4}.
$$
The induced filtrations on each of the simple composition factors start in the following levels: $4$, $3$, $6$, $5$, $4$, and the generation level of the Hodge filtration on each of the local cohomology modules is $4$, $6$, $5$, $4$, respectively.
\end{example}


\section*{Acknowledgments}

We are grateful to Claudiu Raicu for numerous valuable conversations, especially during the early stages of this project while work on \cite{perlmanraicu} was in progress. We also thank Claudiu for pointing out the alternate description of the Hodge filtration discussed in Remark \ref{ClaudiuRemark}. We thank an anonymous referee for carefully reading the manuscript, providing numerous comments, and bringing to our attention an error in the proof of an earlier version of Theorem \ref{main}.

\end{document}